\documentclass[11pt]{amsart}
\usepackage{graphicx}
\usepackage[centertags]{amsmath}
\usepackage{amsfonts}
\usepackage{amssymb}
\usepackage{amsthm}
\usepackage{newlfont}
\newtheorem{theorem}{Theorem}[section]
\newtheorem{corollary}[theorem]{Corollary}

\newtheorem{prop}[theorem]{Proposition}
\theoremstyle{definition}

\theoremstyle{remark}

\newcommand{\id}{\textrm{id}}

\title[Almost product manifolds with a circulant structure]{Riemannian almost product manifolds generated by a circulant structure}
\author{Dimitar Razpopov}
\address{Dimitar Razpopov \\ Department of
Mathematics, Informatics and Physics\\ Faculty of Economics, University of Agriculture\\ 12 Mendeleev blvd, 4000 Plovdiv, Bulgaria \\}
\email{razpopov@au-plovdiv.bg}
\author{Dobrinka Gribacheva}
\address{Dobrinka Gribacheva\\Department of Algebra and Geometry\\ Faculty of Mathematics and Informatics, \\University of Plovdiv Paisii Hilendarski\\ 24 Tzar Asen, 4000 Plovdiv, Bulgaria \\}
\email{dobrinka@uni-plovdiv.bg}

\begin{document}
\begin{abstract}
 A $4$-dimensional Riemannian manifold equipped with a circulant structure, which
is an isometry with respect to the metric and its fourth power is the identity, is considered. The almost product manifold associated with the considered manifold is studied.
The relation between the covariant derivatives of the almost product structure and the circulant structure is obtained.
 The conditions for the covariant derivative of the circulant structure, which imply that an almost product manifold belongs to each of the basic classes of the Staikova-Gribachev classification, are given.
\end{abstract}
\maketitle

\section{Introduction}
The circulant matrices and the circulant structures have application to Vibration analysis, Graph theory, Linear codes, Geometry (for example \cite{13}, \cite{72} and \cite{2}).
Riemannian manifolds equipped with a circulant structure, whose fourth power is the identity were considered in \cite{5} and \cite{1}. In particular case, such manifolds could be Riemannian almost product manifolds.
The systematic development of the theory of Riemannian manifolds $M$ with a metric $g$ and an almost product structure $P$ was started by K.~Yano in \cite{11}. In \cite{nav} A.~M.~Naveira classified the almost product manifolds $(M, P, g)$ with respect to the covariant derivative of $P$.
The Riemannian almost product manifolds $(M, P, g)$ with zero trace of the structure $P$ were classified with respect to the covariant derivative of $P$ by M.~Staikova and K.~Gribachev in \cite{S-G}. The basic classes in this classification are $\mathcal{W}_{1}$, $\mathcal{W}_{2}$ and $\mathcal{W}_{3}$. The class $\mathcal{W}_{0}=\mathcal{W}_{1}\cap \mathcal{W}_{2}\cap W_{3}$ was called the class of Riemannian $P$-manifolds. Our purpose is to obtain characteristic conditions for each of these classes according to the circulant structure.

   In the present paper we consider a $4$-dimensional differentiable manifold $M$ with a Riemannian metric $g$ and a circulant structure $Q$, whose fourth power is the identity and $Q$ acts as an isometry on $g$. This manifold we will denote by $(M, Q, g)$. We study the Riemannian almost product manifold $(M, P, g)$, where $P=Q^{2}$.

   The paper is organized as follows. In Sect. \ref{1}, some necessary facts about considered manifolds $(M, Q, g)$ and $(M, P, g)$ are recalled.
In Sect. \ref{2}, the relation between the covariant derivative of $P$ and the covariant derivative of $Q$ is obtained.

In Sect. \ref{3}, the conditions for the covariant derivative of $Q$, which imply that $(M, P, g)$ belongs to each of the basic classes of the Staikova-Gribachev classification, are given.

\section{Preliminaries}\label{1}
 Let $M$ be a $4$-dimensional Riemannian manifold equipped with a metric $g$ and an endomorphism $Q$ in the tangent space $T_{p}M$ at an arbitrary point $p$ on $M$. Let the coordinates of $Q$ with respect to some basis $\{e_{i}\}$ of $T_{p}M$ form the circulant matrix
   \begin{equation}\label{str-q} Q=\begin{pmatrix}
      0 & 1 & 0 & 0\\
      0 & 0 & 1 & 0 \\
      0 & 0 & 0 & 1\\
      1 & 0 & 0 & 0\\
    \end{pmatrix}.
\end{equation}
Then $Q$ satisfies the equalities
\begin{equation*}
    Q^{4}=\id,\qquad Q^{2}\neq\pm \id.
\end{equation*}
Let the structure $Q$ be compatible with the metric $g$, i.e.
\begin{equation}\label{2.1}
    g(Qx, Qy)=g(x, y).
\end{equation}
Here and anywhere in this work $x, y, z, u$ will stand for arbitrary elements of the algebra of the smooth vector fields on $M$ or vectors in the tangent space $T_{p}M$. The Einstein summation convention is used, the range of the summation indices being always $\{1, 2, 3, 4\}$.

Further, we consider a manifold $(M, Q, g)$ equipped with a metric $g$ and a structure $Q$, which satisfy \eqref{str-q} and \eqref{2.1}.
This manifold is studied in \cite{5} and \cite{1}.

We denote $P = Q^{2}$. In \cite{5} it is noted that the manifold $(M, P, g)$ is a Riemannian manifold with an almost product structure $P$, because $P^{2} = \id $, $P\neq\pm \id$ and $g(Px, Py) = g(x, y)$. Moreover $tr P = 0$. For such manifolds is valid the Staikova-Gribachev classication given in \cite{S-G}.
This classification was made with respect to the tensor $F$ of type $(0,3)$ and the Lee form $\alpha$, which are defined by
\begin{equation}\label{F}
  F(x,y,z)=g((\nabla_{x}P)y,z),\quad \alpha(x)=g^{ij}F(e_{i},e_{j},x).
\end{equation}
Here $\nabla$ is the Levi-Civita connection of $g$, and $g^{ij}$ are the components of the inverse matrix of $g$ with respect to $\{e_{i}\}$.

The basic classes of the Staikova-Gribachev classification are $\mathcal{W}_{1}, \mathcal{W}_{2}$ and $\mathcal{W}_{3}$. Their intersection is the class of Riemannian $P$-manifolds $\mathcal{W}_{0}$.
 A manifold $(M, P, g)$ belongs to each of these classes if it satisfies the following conditions:
\begin{equation}\label{c0}
 \mathcal{W}_{0}:\quad F(x, y, z)=0,
\end{equation}
\begin{equation}\label{c1}
\begin{split}
\mathcal{W}_{1}:\quad F(x,y,z)&=\frac{1}{4}\big((g(x,y)\alpha(z)+g(x,z)\alpha(y)\\&-g(x,Py)\alpha(Pz)-g(x,Pz)\alpha(Py)\big),
\end{split}
\end{equation}
\begin{equation}\label{c2}
\begin{split}
\mathcal{W}_{2}:\quad F(x,y,Pz)+F(y,z,Px)+F(z,x,Py)=0,\quad \alpha(z)=0,
\end{split}
\end{equation}
\begin{equation}\label{c3}
\begin{split}
\mathcal{W}_{3}:\quad F(x,y,z)+F(y,z,x)+F(z,x,y)=0.
\end{split}
\end{equation}

It is well known that $\nabla$ satisfies the equalities:
\begin{equation}\label{f4}
(\nabla_{x}Q)y = \nabla_{x}Qy - Q\nabla_{x}y,
\end{equation}
\begin{equation}\label{fp}
(\nabla_{x}P)y = \nabla_{x}Py - P\nabla_{x}y.
\end{equation}
Let the structure $Q$ of a manifold $(M, Q, g)$ be the covariant constant, i.e. $(\nabla_{x}Q)y=0$. Then, from (\ref{f4}) we obtain successively
\begin{align*}
&\ \nabla_{x}Qy=Q\nabla_{x}y,\quad
\nabla_{x}Q^{2}y=Q\nabla_{x}Qy=Q^{2}\nabla_{x}y,
\end{align*}
thus we get
\begin{equation*}
 \nabla_{x}Py=P\nabla_{x}y.
\end{equation*}
Therefore, from \eqref{fp} it follows
\begin{equation*}
 (\nabla_{x}P)y=0.
\end{equation*}
By using the latter equality and \eqref{F} we find \eqref{c0}. Hence the next theorem is valid.
\begin{theorem}\label{th1.1} If the structure $Q$ of the manifold $(M, Q, g)$ satisfies\\ $\nabla Q=0$, then $(M, P, g)$ belongs to the class $\mathcal{W}_{0}$.
\end{theorem}
As it is known the curvature tensor $R$ of $\nabla$ is determined by $$R(x, y)z=\nabla_{x}\nabla_{y}z-\nabla_{y}\nabla_{x}z-\nabla_{[x,y]}z.$$ The corresponding tensor of type $(0, 4)$ is defined as follows
$$R(x, y, z, u)=g(R(x, y)z,u).$$
\begin{prop} \cite{1}
If the structure $Q$ of the manifold $(M, Q, g)$ satisfies $\nabla Q = 0$, then for the curvature tensor $R$ it is valid
\begin{equation*}
 R(x,y,Qz,Qu)=R(x,y,z,u).
 \end{equation*}
 \end{prop}
We substitute $Qz$ for $z$ and $Qu$ for $u$ in the latter equality, and using Theorem~\ref{th1.1}, we obtain
 \begin{corollary} If the manifold $(M, P, g)$ belongs to $\mathcal{W}_{0}$, then the curvature tensor $R$ satisfies
 \begin{equation*}
 R(x,y,z,u)=R(x,y,Pz,Pu),
 \end{equation*}
 i.e. $R$ is a Riemannian $P$-tensor.
 \end{corollary}

\section{Relation between $F$ and $\bar{F}$}\label{2}

 We consider manifolds $(M, Q, g)$ and $(M, P, g)$, where $P=Q^{2}$. We define a tensor $\bar{F}$ of type $(0,3)$, as follows
\begin{equation}\label{f5}
\bar{F}(x,y,z)=g((\nabla_{x}Q)y,z),\quad \overline{\alpha}(x)=g^{ij}\bar{F}(e_{i},e_{j},x),
\end{equation}
where $\alpha$ is the Lee form associated to $\bar{F}$.
\begin{theorem}\label{th1} For the tensors $F$  on the manifold $(M,P,g)$ and $\bar{F}$ on the manifold $(M,Q,g)$ the following equalities are valid:
\begin{equation}\label{f14}
\bar{F}(x,y,z)+\bar{F}(x,Qy,Qz)=F(x,y,Qz),
\end{equation}
\begin{equation}\label{f15}
\bar{F}(x,y,Q^{3}z)+\bar{F}(x,Qy,z)=F(x,y,z).
\end{equation}
\end{theorem}
\begin{proof}
From \eqref{F} and \eqref{fp}, due to $P=Q^{2}$, we get
\begin{equation*}
 F(x,y,z)=g(\nabla_{x}Py-P\nabla_{x}y,z)=g(\nabla_{x}Q^{2}y-Q^{2}\nabla_{x}y,z),
\end{equation*}
i.e.
\begin{equation*}
F(x,y,z)=g(\nabla_{x}Q^{2}y,z)-g(Q^{2}\nabla_{x}y,z).
\end{equation*}
Then
\begin{equation*}
F(x,y,Qz)=g(\nabla_{x}Q^{2}y,Qz)-g(Q^{2}\nabla_{x}y,Qz),
\end{equation*}
from which, because of \eqref{2.1} we have
\begin{equation}\label{f12}
F(x,y,Qz)=g(\nabla_{x}Q^{2}y,Qz)-g(Q\nabla_{x}y,z).
\end{equation}
From \eqref{f4} and \eqref{f5} we obtain
\begin{equation}\label{14}
\bar{F}(x,y,z)=g(\nabla_{x}Qy,z)-g(Q\nabla_{x}y,z),
\end{equation}
and consequently
\begin{equation}\label{15}
\bar{F}(x,Qy,Qz)=g(\nabla_{x}Q^{2}y,Qz)-g(\nabla_{x}Qy,z).
\end{equation}
Taking the sum of \eqref{14} and \eqref{15} we get
\begin{equation*}
\bar{F}(x,y,z)+\bar{F}(x,Qy,Qz)=g(\nabla_{x}Q^{2}y,Qz)-g(Q\nabla_{x}y,z).
\end{equation*}
Then, having in mind (\ref{f12}), we find \eqref{f14}.
Now we substitute $Q^{3}z$ for $z$ into (\ref{f14}) and using \eqref{2.1}, we find \eqref{f15}.
\end{proof}
Hence the next theorem is valid.
\begin{theorem} The manifold $(M, P, g)$  belongs to $\mathcal{W}_{0}$ if and only if $Q$ satisfies
\begin{equation}\label{fs}
(\nabla_{x}Q)Qy=-Q(\nabla_{x}Q)y.
\end{equation}
\end{theorem}
\begin{proof}
Let $(M, P, g) \in \mathcal{W}_{0}$, i.e. $F = 0$. Then, due to (\ref{f14}), it follows:
\begin{equation*}
\bar{F}(x,Qy,Qz)=-\bar{F}(x,y,z).
\end{equation*}
The latter equality  and (\ref{th1}) imply (\ref{fs}).\\
Vice versa. According to (\ref{f5}) and (\ref{fs}) we find
\begin{equation*}
\bar{F}(x,y,z)+\bar{F}(x,Qy,Qz)=0.
\end{equation*}
Then, due to (\ref{f14}), it follows $F=0$, i.e. $(M, P, g) \in \mathcal{W}_{0}$.
\end{proof}

\section{Properties of $\bar{F}$}\label{3}
\begin{theorem}
For the tensor $\bar{F}$ on $(M, Q, g)$ the following equalities are valid:
\begin{equation}\label{f18}
\bar{F}(x,y,Q^{3}z)+ \bar{F}(x,Qy,z)=\bar{F}(x,z,Q^{3}y)+ \bar{F}(x,Qz,y),
\end{equation}
\begin{equation}\label{f19}
\bar{F}(x,y,z)+ \bar{F}(x,Qy,Qz)+\bar{F}(x,Q^{2}y,Q^{2}z)+ \bar{F}(x,Q^{3}y,Q^{3}z)=0,
\end{equation}
\begin{equation}\label{f20}
\bar{F}(x,y,Qz)=-\bar{F}(x,z,Qy),
\end{equation}
\begin{equation}\label{f21}
\bar{F}(x,y,Q^{3}z)=-\bar{F}(x,Q^{2}z,Qy).
\end{equation}
\end{theorem}
\begin{proof}
It is known that the tensor $F$ determined by \eqref{F} has the properties:
\begin{equation}\label{f16}
F(x,y,z)=F(x,z,y),
\end{equation}
\begin{equation}\label{f17}
F(x,Py,Pz)=-F(x,y,z).
\end{equation}
Equalities (\ref{f15}) and (\ref{f16}) imply (\ref{f18}).\\
From (\ref{f15}) and (\ref{f17}) we get
\begin{equation*}
\bar{F}(x,Qy,z)+ \bar{F}(x,y,Q^{3}z)+\bar{F}(x,Q^{3}y,Q^{2}z)+ \bar{F}(x,Q^{2}y,Qz)=0.
\end{equation*}
In the latter equality we substitute $Qz$ for $z$, and we obtain (\ref{f19}).\\
Further, we substitute $Qz$ for $z$ into (\ref{14}) and we have
\begin{align*}
&\ \bar{F}(x,y,Qz)=g(\nabla_{x}Qy,Qz)-g(\nabla_{x}y,z) \\
&\ =xg(Qy,Qz)-g(Qy,\nabla_{x}Qz)-xg(y,z)+g(y,\nabla_{x}z)\\
&\ =-g(\nabla_{x}Qz,Qy)+g(\nabla_{x}z,y) =-g(\nabla_{x}Qz,Qy)+g(Q\nabla_{x}z,Qy) \\
&\ =-g(\nabla_{x}Qz-Q\nabla_{x}z,Qy)=-\bar{F}(x,z,Qy).
\end{align*}
Therefore we obtain (\ref{f20}).
From (\ref{f20}) directly follows (\ref{f21}).
\end{proof}
Using  (\ref{c1}), \eqref{f5} and (\ref{f14}) we get the following
\begin{theorem} The manifold $(M, P, g)$ belongs to $\mathcal{W}_{1}$ if and only if the tensor $\bar{F}$ on $(M, Q, g)$ satisfies the following conditions:
\begin{align*}
 \bar{F}(x,y,Q^{3}z)+\bar{F}(x,Qy,z)&\ =\frac{1}{4}\big(g(x,y)\alpha(z)+g(x,z)\alpha(y)\\
&\ +g(x,Q^{2}y)\alpha(Q^{2}z)+g(x,Q^{2}z)\alpha(Q^{2}y)\big),
\end{align*}
\begin{equation*}
   \overline{\alpha}(Q^{3}z)+g^{ij}\bar{F}(e_{i}, Qe_{j},z)=\alpha(z).\quad \Box
   \end{equation*}
\end{theorem}
We apply (\ref{f14}) and (\ref{f15}) into (\ref{c2}) and obtain
\begin{equation*}
F(x,y,Pz)=\bar{F}(x,y,Qz)+\bar{F}(x,Qy,Q^{2}z).
\end{equation*}
Therefore we arrive at the following
\begin{theorem} The manifold $(M, P, g)$ belongs to $\mathcal{W}_{2}$ if and only if the tensor $\bar{F}$ on $(M, Q, g)$ satisfies the following condition:
\begin{equation*}
\begin{split}
\bar{F}(x,y,Qz)+\bar{F}(x,Qy,Q^{2}z)+\bar{F}(y,z,Qx)+\bar{F}(y,Qz,Q^{2}x)\\+\bar{F}(z,x,Qy)+\bar{F}(z,Qx,Q^{2}y)=0.\quad \Box
\end{split}
\end{equation*}
\end{theorem}
We apply (\ref{f15}) into (\ref{c3}) and we have
\begin{theorem} The manifold $(M, P, g)$ belongs to $\mathcal{W}_{3}$ if and only if the tensor $\bar{F}$ on $(M, Q, g)$ satisfies the following condition:
\begin{equation*}
\begin{split}
\bar{F}(x,y,Q^{3}z)+\bar{F}(x,Qy,z)+\bar{F}(y,z,Q^{3}x)+\bar{F}(y,Qz,x)\\+\bar{F}(z,x,Q^{3}y)+\bar{F}(z,Qx,Qy)=0.\quad \Box
\end{split}
\end{equation*}
\end{theorem}

\end{document}